\documentclass{amsart}
\usepackage{graphicx} 
\usepackage{amsmath}
\usepackage{mathtools}
\usepackage
{amstext, amscd, setspace, tikz-cd,mathtools, enumerate, graphics, latexsym, siunitx}

\usepackage{pst-node}
\usepackage{tikz-cd} 

\usepackage{amssymb}
\usepackage{hyperref}

\usepackage{PTSerif} 
\usepackage{graphicx}

\usepackage[cmtip,all]{xy}

\newcommand{\F}{{\mathbb F}}

\newcommand{\Q}{{\mathbb Q}}

\newcommand{\GL}{{\rm GL}}
\newcommand{\PGL}{{\rm PGL}}
\newcommand{\SL}{{\rm SL}}
\newcommand{\PSL}{{\rm PSL}}

\newtheorem{theorem}{Theorem}[section]
\newtheorem{corollary}{Corollary}[theorem]

\newtheorem{lemma}[theorem]{Lemma}
\newtheorem{proposition}[theorem]{Proposition}

\theoremstyle{definition}
\newtheorem{definition}[theorem]{Definition}
\newtheorem{example}[theorem]{Example}

\newtheorem{remark}{Remark}[theorem]

\numberwithin{equation}{section}

\title{Right Splitting, Galois Correspondence, Galois Representations and Inverse Galois Problem}
\author{Chandrasheel Bhagwat, Shubham Jaiswal}

\address{Indian Institute of Science Education and Research, Dr.\,Homi Bhabha Road, Pashan, Pune 411008,  INDIA.}
\email{cbhagwat@iiserpune.ac.in, \ jaiswal.shubham@students.iiserpune.ac.in}

\subjclass[2020]{11R32, 12F12, 11F80}
\date{\today}

\begin{document}

\begin{abstract}
  In this article, we realize some groups as Galois groups over rational numbers and finite extension of rational numbers by studying right splitting of some exact sequences, Galois correspondence and algebraic operations on Galois representations.
\end{abstract}

\maketitle

\section{Introduction} The famous `Inverse Galois problem' asks whether every finite group appears as the Galois group of some finite Galois extension of $\Q$. Many families of simple groups are known instances of this problem but the general question is still open.\smallskip

Using Galois representations attached to elliptic curves, Arias-de-Reyna and K{\"o}nig have proved in 
\cite{arias2022locally} that there are infinitely many number fields with  Galois group $\GL_2 (\F_p) $ over $\Q$ that are pairwise linearly disjoint over $\Q(\sqrt{p^\ast})$ (see Remark \ref{result_in_arias2022locally}).\smallskip

In \cite{zywina2023modular}, the case of occurrence of $\PSL_2(\F_p)$ as a Galois group over $\Q$ for all primes $p \geq 5$ was established by Zywina using the results of Ribet \cite{ribet} about the Deligne's Galois representations associated to certain newforms.\smallskip

In this article, we establish some instances of Inverse Galois problem. In Sec.~\ref{groups}, through Galois correspondence and right splitting of some exact sequences of groups, we obtain the main results Theorem~\ref{thm: compositum}, Theorem~\ref{semidirectwithGL2}.
We then apply these general results to the case in \cite{arias2022locally}  and obtain interesting consequences as corollaries which are as follows.

\begin{theorem}
Let $p$ be a prime $\geq 5$. 

\begin{enumerate}
    \item The group $(\SL_2(\mathbb{F}_p)\times \SL_2(\mathbb{F}_p))\rtimes \mathbb{Z}/(p-1)\mathbb{Z}$ (with semidirect product group law as in Thm \ref{semidirectwithGL2}) is realizable as Galois group over $\mathbb{Q}$.   

\item  The group $(\PSL_2(\mathbb{F}_p)\times \PSL_2(\mathbb{F}_p))\rtimes \mathbb{Z}/2\mathbb{Z}$ (with semidirect product group law as in Thm \ref{thm: compositum}) is realizable as Galois group over $\mathbb{Q}$.

 \item  For the case when $p\equiv 3 $ $mod$ $4$, let $H$ be the unique index-$2$ (hence normal) subgroup of $\GL_2(\mathbb{F}_p)$. Then  $(H\times H)\rtimes \GL_2(\mathbb{F}_p)/H$ (with semidirect product group law
as in Thm \ref{semidirectwithGL2}) is realizable as Galois group over $\mathbb{Q}$.

\end{enumerate}
    
\end{theorem}

\smallskip

In Sec.~\ref{Gal-rep}, using the algebraic operation induction on Galois representations and right splitting of some exact sequences of groups, we obtain Thm.~\ref{thm:induced}, Thm.~\ref{thm: M} as the main results. We then apply these general results to the case in \cite{zywina2023modular} and obtain following interesting consequences in Thm.~\ref{thm : zywina case} and Thm.~\ref{genral q}.

\begin{theorem} 
\hfill
\begin{enumerate}
    \item  For a prime $p\geq 5$, $\PSL_2(\mathbb{F}_p)\rtimes \mathbb{Z}/2\mathbb{Z}$ is realizable as Galois Group over $\mathbb{Q}$ (for semidirect product group law in Thm \ref{thm:induced}). The group obtained here is not isomorphic to $\PGL_2(\mathbb{F}_p)$.

      \item The group $\PSL_2(\mathbb{F}_q)\rtimes \mathbb{Z}/2\mathbb{Z}$ is realisable as Galois Group over $\mathbb{Q}$ (for semidirect product group law in Theorem \ref{thm:induced}) for (1) $q=p$ for $p\geq 5$, (2) $q=p^2$ for $p\equiv \pm 2 \ mod \ 5,\ p\geq 7$, (3) $q=p^2$ for $p\equiv \pm 3 \ mod \ 8,\ p\geq 5$, (4) $q=p^3$ for odd prime $p\equiv \pm 2, \pm 3, \pm 4, \pm 6 \ mod \ 13$ and (5) $q=5^3,3^5,3^4$.

     \end{enumerate}
\end{theorem}
\smallskip

Then for the algebraic operations direct sum and tensor products on Galois representations we obtain an unanticipated and interesting result Prop \ref{tensor direct sum}.\smallskip

Then by using the algebraic operations direct sum and tensor products on Galois representations and right splitting of some exact sequences of groups, we obtain Thm.~\ref{thm : L} as the main result. We then apply this general result to the case in \cite{zywina2023modular} and obtain that $\PSL_2(\mathbb{F}_p)\times \mathbb{Z}/2\mathbb{Z}$ is realizable as Galois Group over $\mathbb{Q}$ for $p\geq 5$.\smallskip

Right Splitting of exact sequences of groups is the common thread that runs through the whole article.

\medskip

\section{Groups, Galois Correspondence, Right Splitting and Inverse Galois Problem} \label{groups}

In this section, we describe some group theoretic results and then through these results as well as Galois correspondence and right splitting of some exact sequences, we obtain some general Galois theoretic results. Then we apply these general results to the cases described in \cite{arias2022locally} and obtain interesting consequences. Although some of the results related to groups and Galois correspondence may be known to experts, but we include them in our discussion for the sake of completeness.\smallskip

\subsection{Right splitting and group theoretic results}

\hfill

\medskip

Recall that if $H$ is a normal subgroup of $G$
such that $G/H$ is abelian then $[G, G] \subset H$. Following is a simple lemma where the converse follows from structure theorem for abelian groups.

\begin{lemma}
     Let $G$ be a finite group. Then the condition $G/[G,G]$ is cyclic of order $m$ is equivalent to the condition that for any $n|m$, $G$ has a unique index-$n$ normal subgroup such that the quotient is abelian (the quotient with that subgroup is in fact cyclic). This unique subgroup is given by $$H=\{x\in G \ |\ \pi(x) \text{ is an}\ n\text{-th power in}\ G/[G,G]\}$$ where $\pi$ is the quotient homomorphism $G\rightarrow  G/[G,G]$.
\end{lemma}

   
\smallskip

Let $\mathbb{F}_q$ be finite field of $q$ elements where $q=p^r$ for an odd prime $p$ and $r\in \mathbb{N}$. We know that $[\GL_2(\mathbb{F}_q),\GL_2(\mathbb{F}_q)]=SL_2(\mathbb{F}_q)$ and $\GL_2(\mathbb{F}_q)/SL_2(\mathbb{F}_q)\cong \mathbb{F}_q^{\times}$ through the surjective determinant map and $\mathbb{F}_q^{\times}$ is cyclic. By applying similar argument to $det$ map as applied to $\pi$ map in above lemma, we get the following. 

\begin{corollary}
    \label{cor: unique}
     Let $n\in \mathbb{N}$. If $q\equiv 1$ $mod$ $n$, then $\GL_2(\mathbb{F}_q)$ has a unique index-$n$ normal subgroup such that the quotient is abelian. (The quotient with that subgroup is in fact cyclic). This unique subgroup is given by $$H=\{x\in \GL_2(\mathbb{F}_q)\ | \ det(x)\ \text{ is an}\ n\text{-th  power  in } \ \mathbb{F}_q^{\times}\}.$$
\end{corollary}


\bigskip

\begin{corollary}
\label{cor: sl2}
\hfill

\begin{enumerate}
    \item 

    Let $G$ be a group. If $[G,G]$ has index $m$ in $G$, then $[G,G]$ is the unique index-$m$ subgroup of $G$ such that the quotient is abelian. 

       \item $SL_2(\mathbb{F}_q)$ is the unique index-$(q-1)$ subgroup of  $\GL_2(\mathbb{F}_q)$ such that the quotient is abelian.

       \end{enumerate}
\end{corollary}

\smallskip

For $a,b\in \mathbb{N}$, we denote $gcd(a,b)$ by $(a,b)$.

\begin{proposition}
    \label{prop: beautiful}
    
 Let $n\in \mathbb{N}$ and let $G$ be a group such that $G/[G,G]$ is cyclic of order $m$. Let $n|m$ and let $H$ be the unique index-$n$ normal subgroup of $G$ such that the quotient is abelian. Then the following holds.
 
 \smallskip
 \begin{enumerate}
     \item If there exists $x\in G$ such that $x^n=1$ and $\{1,x,x^2,...,x^{n-1}\}$  is a 
      set of representatives for $H$-cosets in $G$, then $$(n,m/n) =1.$$

\smallskip
     
     \item If $(n,m/n)=1$ and the exact sequence \begin{equation}\label{first exact}
         1\rightarrow [G,G]\rightarrow G \rightarrow G/[G,G]\rightarrow 1\end{equation} is right split, then
      there exists $x\in G$ such that $x^n=1$ and $\{1,x,x^2,...,x^{n-1}\}$
    is a set of representatives for $H$-cosets in $G$.
     \smallskip
     
     \item Let $n=m$ (hence $[G,G]$ is the unique index-$m$ subgroup of $G$ such that quotient is abelian). Then the exact sequence \ref{first exact} is right split if and only if there exists $x\in G$ such that $x^m=1$ and $\{1,x,x^2,...,x^{m-1}\}$ is a set of representatives for $[G,G]$-cosets in $G$.
 \end{enumerate}

\end{proposition} 
    
\begin{proof} \hfill

 \begin{enumerate}
\item  Suppose we have a set of representatives for $H$-cosets in $G$ of the form $\{1,x,x^2,...,x^{n-1}\}$ with $x\in G$ such that $x^n=1$. Let $y$ be a generator of cyclic group $G/[G,G]$ and let $0\leq l\leq (m-1)$ be such that $\pi (x)=y^l$. As $x^n = 1$, we have $1 = \pi(x)^n = y^{ln}$ and hence $m \mid ln$. Thus we have 
$$j = ln/m \in \mathbb Z \ \text{and}\ 1\leq j\leq (n-1).$$

Now assume on the contrary that $(n,m/n)\neq 1$. Then for $k=n/(n,m/n)\leq (n-1)$, 
we have $n\mid km/n$, hence $n\mid kl$, and therefore $\pi(x^k)=y^{lk}$ is an $n$-th power in $G/[G,G]$, hence $x^k\in H$. Hence we get a contradiction. Similarly we can show that $(j,n)=1$. \smallskip

\item If $(n,m/n)=1$, then we clearly have that $n\nmid km/n$ for all $1\leq k\leq (n-1)$. Consider the quotient map $\pi :G\rightarrow  G/[G,G]$. Since given exact sequence \ref{first exact} is right split, we have an injective homomorphism $\iota :  G/[G,G] \rightarrow G$ such that $\pi \circ \iota = id_{G/[G,G]}$. Let $x=\iota(y^{jm/n}) \in G$ where $1\leq j\leq (n-1)$ and $(j,n)=1$. Hence $\pi(x)= y^{jm/n}$. Since $\iota$ is injective, $x$ will be an element of order $n$.\  Now since  $n\nmid km/n$ for all $1\leq k\leq (n-1)$ and since $(j,n)=1$, we have $n\nmid jkm/n$ for all $1\leq k\leq (n-1)$. Hence for all $1\leq k\leq (n-1)$, $\pi(x^k)=y^{jkm/n}$ is not an $n$-th power in $G/[G,G]$. Hence $x^k\not \in H$. Therefore we get the required set of representatives for $H$-cosets in $G$. \smallskip

\item If the exact sequence \ref{first exact} is right split, then by (2), we are done. The converse is easy.
 \end{enumerate}
 \vspace{-.5 cm}
 \end{proof}
\smallskip

 \begin{remark} 
 \hfill
 \begin{enumerate}
     \item If $G$ is abelian, then it satisfies the hypothesis of above proposition if and only if it is cyclic.
 
 \item   The condition that $G/[G,G]$ is cyclic is important in statement (3) above. If $G$ is a non-cyclic abelian group, then $[G,G]=1$ and so exact sequence $1\rightarrow 1\rightarrow G\rightarrow G\rightarrow 1$ is clearly right split, but we can never have required representatives for $[G,G]$-cosets because of non-cyclicity of $G$. 
 
 \end{enumerate}
 \end{remark}
 \smallskip

\begin{corollary}
    \label{cor : beautiful }
Let $n\in \mathbb{N}$. Consider $q$ such that $q\equiv 1$ mod $n$. Let $H$ be the unique index-$n$ subgroup of $\GL_2(\mathbb{F}_q)$ such that the quotient is abelian. Then the following are equivalent. 

\begin{itemize}
\item $(n,(q-1)/n)=1$. \smallskip

\item There exists a set of representatives for $H$-cosets in $\GL_2(\mathbb{F}_q)$ of the form $\{1,x,x^2,...,x^{n-1}\}$ with $x\in \GL_2(\mathbb{F}_q)$ such that $x^n=I$.

\end{itemize}
 \smallskip

\end{corollary}
    
\begin{proof}
Let $G=\GL_2(\mathbb{F}_q)$. We know that below exact sequence given by the determinant map is right split. 
\begin{equation} 
1 \rightarrow [G,G] \rightarrow G \xrightarrow{det} \mathbb{F}_q^{\times}\rightarrow 1.\end{equation} By applying similar argument to $det$ map as applied to $\pi$ map in previous proposition, we are done.
 \end{proof}
\smallskip


The following follows from Dirichlet's theorem on primes in arithmetic progression.

\begin{proposition}
    
For given $n,r\in \mathbb{N}$ such that $(r,n)=1$, there are infinitely many primes satisfying the two conditions
$$p^r   \equiv 1\  mod\ n\ \text{and}\ (n,(p^r - 1)/n)  = 1.
$$
\end{proposition}

\subsection{Galois correspondence, right splitting and Galois theoretic results} We discuss a result here that establishes the occurrence of certain semidirect product of groups as a Galois group.\smallskip

\begin{theorem}

\label{thm: compositum}
    Let $l\geq 2$. Let $K$ be a  finite extension of $\mathbb{Q}$ and let $E_1,E_2,..., E_l$ be finite Galois extensions of $\mathbb{Q}$ contained in $\bar{\mathbb{Q}}$. Let $G_i$ be the Galois group of $E_i$ over $\mathbb{Q}$ for all $i$. Suppose  \smallskip
     \begin{enumerate}  
     
     \item $E_1\cdots E_k \cap  E_{k+1} = K$ for all $1\leq k\leq l-1$,\smallskip

     \item for every $i$, there exists a set of representatives of $H_i$-cosets in respective $G_i$ that is closed under multiplication, where $H_i$ is Galois group of $E_i$ over $K$. \smallskip
     
     \end{enumerate}
     
     Then for each $1\leq i \leq l$, $H_i$ is a normal subgroup of respective $G_i$ and the group $(H_1\times H_2\times \dots \times H_l)\rtimes G_1/H_1$ (for some semidirect product group law) is realizable as the Galois Group of $E_1\cdots E_l$ over $\mathbb{Q}$.
\end{theorem}

\begin{proof}
     We will prove by induction that ${\rm Gal}(E_1E_2\cdots E_k/K)\cong (H_1\times H_2\times \dots \times H_k)$ and ${\rm Gal}(E_1E_2\cdots E_k/\mathbb{Q})\cong (H_1\times H_2\times \dots \times H_k)\rtimes G_1/H_1$ (for some semidirect product group law) for all $2\leq k\leq l$.\smallskip
     
     For base case $k=2$, see the following diagram. \smallskip
   
  Since, $E_1,E_2$ are Galois over $\mathbb{Q}$ and are contained in $\bar{\mathbb{Q}}$, $E_1E_2$ and $E_1\cap E_2$ are defined and are Galois over $Q$ and $K=E_1\cap E_2$. We have ${\rm Gal}(E_i/\mathbb{Q})=G_i$ and ${\rm Gal}(E_i/K)=H_i$ for $i=1,2$. Since $K$ is Galois over $\mathbb{Q}$, $H_i$s are normal in respective $G_i$s. ${\rm Gal}(K/\mathbb{Q})\cong G_1/H_1 \cong G_2/H_2$. Fix an isomorphism $\theta$ from $G_2/H_2$ to $G_1/H_1$.  
    \smallskip 
    
     
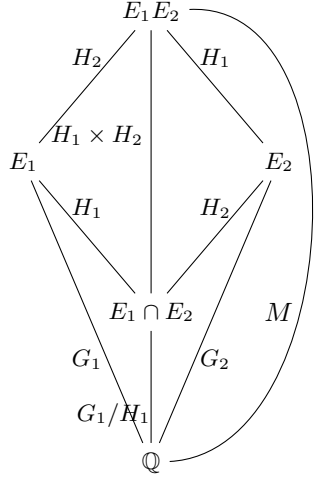
\begin{center}
     \begin{tikzpicture}\label{diagram-compositum}
   \node (Q0) at (0,-2) {\small $\mathbb{Q}$};
    \node (Q1) at (0,0) {\small $E_1\cap E_2$};
    \node (Q2) at (1.7,2) {\small $E_2$};
    \node (Q3) at (0,4) {\small $E_1 E_2$};
    \node (Q4) at (-1.7,2) {\small $E_1$};
    \draw (Q1)--(Q2) node [pos=0.5, above,inner sep=0.25cm] {\small $H_2$};
    \draw (Q1)--(Q4) node [pos=0.5, above,inner sep=0.25cm] {\small $H_1$};
    \draw (Q3)--(Q4) node [pos=0.5, above ,inner sep=0.25cm] {\small $H_2$};
    \draw (Q2)--(Q3) node [pos=0.5, above ,inner sep=0.25cm] {\small $H_1$};
    \draw (Q1)--(Q3) node [pos=0.6, left,inner sep=0.1cm] {\small $H_1\times H_2$};
    \draw (Q0)--(Q1) node [pos=0.25, left,inner sep=0.00cm] {\small  $G_1/H_1$};
    \draw (Q4)--(Q0) node [pos=0.5, below,inner sep=0.5cm] {\small $G_1$};
    \draw (Q2)--(Q0) node [pos=0.5, below,inner sep=0.5cm] {\small $G_2$};
     \draw  (Q3) to [out=3, in=3](Q0) node [pos=.75, right, inner sep= 1.5  cm]{$M$};
    \end{tikzpicture}
    
    \end{center}
 \smallskip
    
    Let ${\rm Gal}(K/\mathbb{Q})=\{x_i H_1| 1\leq i \leq n \}$ where $x_iH_1=\theta(y_iH_2)$ for all $i$. Here $\{x_i\in G_1|1\leq i\leq n\}$ and $\{y_i\in G_2|1\leq i\leq n\}$ are sets of representatives of $H_1$ and $H_2$ cosets in $G_1$ and $G_2$ respectively, which are closed under multiplication.\smallskip
    
    By Galois Correspondence \cite[Theorem 2.1]{conrad2023galois} and \cite[Theorem 2.6]{conrad2023galois},
    we have ${\rm Gal}(E_1 E_2/E_1)\cong H_2$, ${\rm Gal}(E_1 E_2/E_2)\cong H_1$ and ${\rm Gal}(E_1 E_2/K)\cong H_1\times H_2$.\smallskip
    
    Let ${\rm Gal}(E_1 E_2 /\mathbb{Q})=M$. We have 
    \[
    \begin{split}
    |M| & = [E_1 E_2 : K] \ [K:\mathbb{Q}]  = |H_1\times H_2| \ |G_1/H_1| \\
    & = |H_1\times H_2| \ |G_2/H_2|=|G_1| \ |H_2|=|G_2|\ |H_1|. 
    \end{split}
    \]
    
    Consider map $\sigma :M\rightarrow G_1\times G_2$ given by $\sigma (g)=(g_1,g_2)$, where $g_i=g|_{E_i}$ for $i=1,2$. The map $\sigma$ is clearly a well defined injective group homomorphism (see \cite[Theorem 1.1]{conrad2023galois}). \smallskip
    
    If $g\in M$ and $g|_K=x_iH_1=g_1|_{K}$, then $(x_i^{-1}g_1)|_{K}=id_{K}$. Hence $x_i^{-1}g_1\in H_1$, thus $g_1\in x_iH_1$. Similarly, by $g|_K=x_iH_1=\theta(g_2|_{K})=\theta(y_iH_2)$, $g_2\in y_iH_2$.\smallskip
    
    Hence ${\rm Image}(\sigma)\subset \bigsqcup \limits_{i=1}^n (x_iH_1\times y_iH_2)$. Since $|\bigsqcup \limits_{i=1}^n (x_iH_1\times y_iH_2)|=|H_1\times H_2||G_1/H_1|=|M|$ and $\sigma$ is injective, we get ${\rm Image}(\sigma)= \bigsqcup \limits_{i=1}^n (x_iH_1\times y_iH_2)$. Hence, $M\cong \bigsqcup \limits_{i=1}^n (x_iH_1\times y_iH_2)$.  \smallskip
    
     Consider a map $\psi:\bigsqcup \limits_{i=1}^n (x_iH_1\times y_iH_2)\rightarrow (H_1\times H_2)\rtimes G_1/H_1$ given by 
    $\psi((x_ih_1,y_ih_2))=((h_1,h_2),x_iH_1)$ where we define a semidirect product group law for $(H_1\times H_2)\rtimes G_1/H_1$ by 
    $$((h_1,h_2),x_iH_1).((k_1,k_2),x_jH_1)=((x_j^{-1}h_1x_jk_1, y_j^{-1}h_2y_jk_2),x_ix_jH_1).$$ 
    
    This group law is well defined and associative since $H_i$ are normal in $G_i$ for $i=1,2$ and sets of their respective coset representatives are closed under multiplication. We observe that $$(x_ih_1,y_ih_2).(x_jk_1,y_jk_2)=(x_ih_1x_jk_1,y_ih_2y_jk_2)=((x_ix_j)(x_j^{-1}h_1x_jk_1),(y_iy_j)(y_j^{-1}h_2y_jk_2)).$$
    
     Hence we conclude that $\psi$ is a group isomorphism and hence $M\cong (H_1\times H_2)\rtimes G_1/H_1$.  \smallskip
    
    Alternatively, consider the sequence \begin{equation}
        1\rightarrow (H_1\times H_2)\xrightarrow{i} \bigsqcup \limits_{i=1}^n (x_iH_1\times y_iH_2) \xrightarrow{\pi} G_1/H_1\rightarrow 1\end{equation} where $i(h_1,h_2)=(x_1h_1,y_1h_2)$ ($x_1\in H_1,y_1\in H_2$) and $\pi(x_i h_1,y_i h_2)=x_i H_1$. Because of our multiplicatively closed assumption, $x_1$ and $y_1$ are identities of $G_1$ and $G_2$ respectively. Hence $i$ is injective. Also $\pi$ is surjective and $ker(\pi)=Image(i)$. Thus the sequence is exact.\smallskip
    
 Now consider $\iota : G_1/H_1 \rightarrow \bigsqcup \limits_{i=1}^n (x_iH_1\times y_iH_2)$ with $\iota(x_i H_1)=(x_i,y_i)$. The map $\iota$ is clearly a group homomorphism because of our multiplicatively closed assumption and also $\pi \circ \iota = id_{G_1/H_1}$. Hence the sequence is right split. Thus $M\cong (H_1\times H_2)\rtimes G_1/H_1$.
 \smallskip

     Now assume that the statement is true for $k=m$, where $2\leq m\leq l-1$, that is, 
     $${\rm Gal}(E_1E_2\cdots E_m/K)\cong (H_1\times H_2\times \dots \times H_m),$$ 
     $$\text{and}, {\rm Gal}(E_1E_2\cdots E_m/\mathbb{Q})\cong (H_1\times H_2\times \dots \times H_m)\rtimes G_1/H_1,$$
      we will prove the statement for $k=m+1$. \smallskip
     
Let $F_1=E_1\cdots E_m$ and $F_2=E_{m+1}$. Hence
\[      
\begin{split}
{\rm Gal}(F_1/\mathbb{Q}) & = {\rm Gal}(E_1E_2\cdots E_m/\mathbb{Q})\cong (H_1\times H_2\times \dots \times H_m)\rtimes G_1/H_1,\\      
 {\rm Gal} (F_2/\mathbb{Q})& ={\rm Gal}(E_{m+1}/\mathbb{Q})\cong G_{m+1},\\
  {\rm Gal}(F_1/K) &= {\rm Gal}(E_1E_2\cdots E_m/K)\cong (H_1\times H_2\times \dots \times H_m), \\
  {\rm Gal}(F_2/K) &= {\rm Gal}(E_{m+1}/K)\cong H_{m+1}\ \text{since} F_1\cap F_2 =E_1\cdots E_m \cap  E_{m+1} = K.
 \end{split}
\]       

Hence $F_1$ and $F_2$ satisfy conditions of base case and thus we have
 \[
 \begin{split}
 {\rm Gal}(E_1E_2\cdots E_{m+1}/K) & = {\rm Gal}(F_1F_2/K) \cong {\rm Gal}(F_1/K)\times {\rm Gal}(F_2/K)\\
 & \cong  (H_1\times H_2\times \dots \times H_m)\times H_{m+1} = H_1\times H_2\times \dots \times H_{m+1}
 \end{split}
         \]

\text{and},
\[
 \begin{split} 
  {\rm Gal}(E_1E_2\cdots E_{m+1}/\mathbb{Q}) & = {\rm Gal}(F_1F_2/\mathbb{Q}) \cong {\rm Gal}(F_1F_2/K) \rtimes G_1/H_1 \\
  & \cong (H_1\times H_2\times \dots \times H_{m+1})\rtimes G_1/H_1. \end{split}
  \]

  \vspace{-.6 cm}
  \end{proof}

\begin{remark} \label{remark proof} \hfill
 \begin{enumerate}
     \item 
   If $[G:H]$ is finite, then
 a set of representatives for $H$-cosets in $G$ satisfies the condition of closed under multiplication if and only if it is a subgroup of $G$.
 
 In particular,  A set of representatives $\{1,x,x^2,...,x^{n-1}\}$ for $H$-cosets in $G$ which forms a cyclic subgroup of $G$, satisfies the condition. We get a criterion for existence of such a set of representatives  in Proposition~\ref{prop: beautiful}.

\item
 The second condition assumed in the theorem is equivalent to exact sequences \begin{equation} 1\rightarrow H_i \rightarrow G_i \rightarrow G_i/H_i \rightarrow 1\end{equation} being right split, that is $G_i \cong H_i \rtimes G_i/H_i$ for some semidirect product group law. 
  \end{enumerate}
 
 \end{remark}

\begin{remark}
    
 Let $d_i$ be distinct primes for $i=1,2$ and $d_3=d_1d_2$. Let $K=\mathbb{Q}$ and consider quadratic extensions $E_i=\mathbb{Q}\sqrt{d_i}$ for $1\leq i \leq 3$. \smallskip

 Clearly $E_i\cap E_j=K$ for all $i\neq j$, but we have $E_iE_j\cap E_k=E_k\neq K$ where $i,j,k$ are a permutation of $1,2,3$.
So the assumed condition in the above theorem is not always satisfied and hence it is important.
\end{remark}

\begin{remark}
    
 In the above theorem, we could have assumed a symmetric but stronger condition \\ $E_1...E_{i-1}E_{i+1}...E_{n} \cap E_i=K$ for all $1\leq i\leq n$ which implies the condition that we have assumed $E_1...E_k \cap  E_{k+1} = K$ for all $1\leq k\leq n-1$ since $E_1...E_k \cap  E_{k+1}\subset E_1...E_kE_{k+2}...E_n \cap  E_{k+1}=K$. But it would have made our theorem weaker. (Note: Condition for $i=1$ is not required, it is written for symmetry).
\end{remark}

\begin{remark} \hfill
   \begin{enumerate}

   \item 
   Under the conditions of above theorem, we have $(H_1\times H_2\times \dots \times H_l)\rtimes G_1/H_1\xhookrightarrow{}  G_1 \times G_2\times \dots \times G_l$ from \cite[Theorem 1.1]{conrad2023galois}. 
       \item If all $G_i$s are abelian in above theorem that is each $E_i$ is an abelian extension of $\mathbb{Q}$, then the compositum is an abelian extension over $\mathbb{Q}$ and the semidirect product that we defined is in fact the direct product.\smallskip
   \item However to the best of our knowledge, one doesn't know in general whether $(H_1\times H_2\times \dots \times H_l)\times G_1/H_1$ is realizable as a Galois group over $\mathbb{Q}$.
   \end{enumerate}
\end{remark}

We now describe an independently interesting consequence of the Galois Correspondence (see  \cite[Theorem 2.1]{conrad2023galois} and \cite[Theorem 2.6]{conrad2023galois}.) \smallskip

\begin{proposition} 
    Let finite groups $H_1,H_2,..., H_n$ be Galois groups over $K$ of extensions $E_1,E_2,..., E_n$ respectively which are contained in $\bar{\mathbb{Q}}$. Fix $k\leq n-1$ and suppose $E_{i_1}\cdots E_{i_j} \cap  E_{i_{j+1}} = K$ for all $1\leq j\leq k-1$ where $i_l$ are distinct elements from $1$ to $k$. Then the following statements are equivalent.  \smallskip
    \begin{enumerate}
        \item  $E_1\cdots E_k \cap  E_{k+1} = K$.  
         \smallskip
        \item $E_{i_1}\cdots E_{i_k} \cap E_{i_{k+1}}=K$ where $i_l$ are distinct elements from $1$ to $k+1$. 
         \smallskip
        \item $E_1\cdots E_{i-1}E_{i+1}\cdots E_{k+1} \cap  E_1\cdots E_{j-1}E_{j+1}\cdots E_{k+1} = E_1\cdots E_{i-1}E_{i+1}\cdots E_{j-1}E_{j+1}\cdots E_{k+1}$ for any $1\leq i < j \leq k+1$.

         \end{enumerate}

\end{proposition}

\begin{proof} Because of given conditions and induction we get, ${\rm Gal}(E_{i_1}\cdots E_{i_j}/K)\cong H_{i_1}\times...\times H_{i_j}$ for all $j\leq k$ where $i_l$ are distinct elements from $1$ to $k$
    and ${\rm Gal}(E_{i_1}\cdots E_{i_j}/E_{i_1}\cdots E_{i_{m-1}}E_{i_{m+1}}\cdots E_{i_j})\cong H_{i_m}$ for all $j\leq k$ and $1\leq m\leq j$ where $i_l$ are distinct elements from $1$ to $k$.
 \smallskip

\noindent Equivalence of (1) and (2): For $i_l$ distinct elements from $1$ to $k+1$,\smallskip

$E_1...E_k \cap  E_{k+1} = K$.\\ $\iff$ ${\rm Gal}(E_{1}\cdots E_{k+1}/K)\cong H_{1}\times\dots \times H_{k+1}$.\\
$\iff$  $E_{i_1}\cdots E_{i_k} \cap E_{i_{k+1}}=K$. \smallskip

\noindent Equivalence of (1) and (3): For any $1\leq i < j \leq k+1$, let 
\[
\begin{split}
D&=E_{1}\cdots E_{k+1}, \\
 D_i &= E_{1}\cdots E_{i-1}E_{i+1}\cdots E_{k+1},\\
  D_j &= E_{1}\cdots E_{j-1}E_{j+1}\cdots E_{k+1},\\
  D_{ij} &= E_1...E_{i-1}E_{i+1}...E_{j-1}E_{j+1}...E_{k+1}.
  \end{split}
  \]

Now since ${\rm Gal}(D_i / D_{ij})\cong H_{j}$ and ${\rm Gal}(D_j /D_{ij})\cong H_{i}$, we have by \cite[Theorem 1.1]{conrad2023galois}, 
$${\rm Gal}(D/D_{ij}) \xhookrightarrow{} H_i\times H_j.$$

We also have
\[
\begin{split}
{\rm Gal}(D/K)/ {\rm Gal}(D/D_{ij}) & \cong {\rm Gal}(D_{ij}/ K) \\
& \cong H_1\times \dots \times H_{i-1} \times H_{i+1}\times \dots \times H_{j-1} \times H_{j+1} \times \dots \times H_{k+1}.
\end{split}
\]

\noindent Hence,
\[
\begin{split}
& {\rm Gal}(D/K)  \cong  H_{1}\times \dots \times H_{k+1}\\
 \implies &  {\rm Gal}(D/D_{ij})\cong H_i\times H_j \\
 \implies & D_i \cap  D_j  = D_{ij}.
\end{split}
\]

\noindent Conversely, 
\[
\begin{split}
& D_i \cap  D_j= D_{ij} \\ 
 \implies & D_i \cap  E_i \subset D_i  \cap  D_j = D_{ij}\\ 
 \implies & D_i \cap  E_i \subset D_{ij} \cap E_i=K \\
\implies & D_i \cap  E_i =K \\ 
\implies & {\rm Gal}(D/K)\cong H_{1}\times \cdots \times H_{k+1}.
 \end{split}
\]
\vspace{-.8 cm}

\end{proof}

\smallskip

\subsection{Application to cases in work of Arias-de-Reyna \& K{\"o}nig}

\begin{theorem}\label{semidirectwithGL2}

  Let $n\in \mathbb{N}$. For $q$ suppose the following hold. 

  \begin{enumerate}
      \item  $q\equiv 1$ $mod$ $n$ and $(n,(q-1)/n)=1$. 
  
  \item The group $\GL_2(\mathbb{F}_q)$ is realizable as Galois group over $\mathbb{Q}$ of extensions $E_1,E_2$ which are contained in $\bar{\mathbb{Q}}$ such that $E_1\cap E_2=K$ is a degree $n$ abelian extension of $\mathbb{Q}$.
  
  \end{enumerate}
  
  Then  $(H\times H)\rtimes \GL_2(\mathbb{F}_q)/H$ (with semidirect product group law as in Theorem \ref{thm: compositum}) is realizable as Galois group over $\mathbb{Q}$ where $H$ is the unique index-$n$ subgroup of $\GL_2(\mathbb{F}_q)$ such that the quotient is abelian. 
\end{theorem}
\begin{proof}
     We have for $l=2$, $E_1,E_2$ satisfying conditions of Theorem \ref{thm: compositum} with $G_1=G_2=\GL_2(\mathbb{F}_q)$, and $[G_1:H_1] =n$ and $G_1/H_1$ is abelian group and $q\equiv 1$ $mod$ $n$. Therefore $H_1=H_2=H$. Since $q\equiv 1$ $mod$ $n$ and $(n,(q-1)/n)=1$, from Corollary \ref{cor : beautiful }, we get the required set of representatives for $H$-cosets in $\GL_2(\mathbb{F}_q)$.\end{proof}

\begin{remark}
    
 Because of the conditions on $p$ and $n$ in above theorem, we actually get that $K$ is a cyclic extension (not just abelian).
 \end{remark}

 \smallskip

 If extensions $E_1,E_2$ of $\mathbb{Q}$ contained in $\bar{\mathbb{Q}}$, are linearly disjoint over an extension $K$ of $\mathbb{Q}$, then $E_1\cap E_2=K$.
     We get from \cite[Remark 4.4]{arias2022locally} and \cite[Corollary 4.3]{arias2022locally}, that there are two Galois extensions $L_1,L_2$ over $\mathbb{Q}$ with Galois group $\GL_2(\mathbb{F}_p)$ which are linearly disjoint over $\mathbb{Q}(\zeta_p)$.

\begin{corollary}\label{sl2 semi}
For a prime $p\geq 5$,  $(SL_2(\mathbb{F}_p)\times SL_2(\mathbb{F}_p))\rtimes \mathbb{Z}/(p-1)\mathbb{Z}$ (with semidirect product group law as in above theorem) is realizable as Galois group over $\mathbb{Q}$.    
\end{corollary}

\begin{proof}
 From \cite[Remark 4.4]{arias2022locally} and \cite[Corollary 4.3]{arias2022locally},  we get $E_1,E_2$ satisfying conditions of previous theorem with $n=p-1$, $K=\mathbb{Q}(\zeta_p)$ and $\GL_2(\mathbb{F}_p)/H$ is cyclic group of order $p-1$. Moreover from Corollary \ref{cor: sl2} (2), we have $H=SL_2(\mathbb{F}_p)$. 
\end{proof}\smallskip

From previous corollary and proof of Theorem \ref{thm: compositum} we have the following.

\begin{corollary}
    For a prime $p\geq 5$, $SL_2(\mathbb{F}_p)$ and $SL_2(\mathbb{F}_p)\times SL_2(\mathbb{F}_p)$ are realizable as Galois groups over $\mathbb{Q}(\zeta_p)$. 
\end{corollary}

 \begin{remark}\label{result_in_arias2022locally}
It is proved in \cite[Theorem 1.1]{arias2022locally} that for a prime $p \geq 5$, there are infinitely many locally cyclic Galois extensions
of $\mathbb{Q}$ with Galois group $\GL_2(\mathbb{F}_p)$, which are pairwise linearly disjoint over $\mathbb{Q}(\sqrt{p^*})$ where $p^*=(-1)^{(p-1)/2} p$.
  
 \end{remark}

\begin{corollary}
   For a prime $p\geq 5$ with $p\equiv 3 $ $mod$ $4$, let $H$ be the unique index-$2$ (hence normal) subgroup of $\GL_2(\mathbb{F}_p)$. Then  $(H\times H)\rtimes \GL_2(\mathbb{F}_p)/H$ (with semidirect product group law
as in above theorem) is realizable as Galois group over $\mathbb{Q}$.
\end{corollary} 
\begin{proof}
     From \cite[Theorem 1.1]{arias2022locally}, we get $E_1,E_2$ satisfying conditions of previous theorem with $n=2$. Since $p\equiv 1$ $mod$ $2$, the conditions $p\equiv 3$ $mod$ $4$ and $(2,(p-1)/2)=1$ are equivalent. Hence we are done.
\end{proof}

\begin{remark}
    
 If $p\equiv 1 $ $ mod$ $4$ then $(2, (p-1)/2)=2$. Hence by Corollary \ref{cor : beautiful } there is no $x\in \GL_2(\mathbb{F}_p)$ such that $x\in H$ and $order(x)= 2$. 
\end{remark}

Now for $p\equiv 3$ $mod$ $4$ we have $p^* = -p$. From previous corollary and proof of Theorem \ref{thm: compositum} we have the following.

\begin{corollary}
  \label{cor: H cross H}
     For a prime $p\geq 5$ with $p\equiv 3$ $ mod$ $4$, let $H$ be the unique index-$2$ subgroup of $\GL_2(\mathbb{F}_p)$. Then $H$ and $H\times H$ are realizable as Galois groups over $\mathbb{Q}\sqrt{-p}$. 
\end{corollary}\smallskip


 \begin{proposition}
   \label{prop : psl2}  

   For a prime $p\geq 5$,  $(PSL_2(\mathbb{F}_p)\times PSL_2(\mathbb{F}_p))\rtimes \mathbb{Z}/2\mathbb{Z}$ (with semidirect product group law as in Theorem \ref{thm: compositum}) is realizable as Galois group over $\mathbb{Q}$.
 \end{proposition}

\begin{proof}
Consider $H=\{x\in \GL_2(\mathbb{F}_p) \ |\  det(x)$ is a square in $\mathbb{F}_p^{\times}\}$, which is the unique index-$2$ subgroup of $\GL_2(\mathbb{F}_q)$ by Corollary \ref{cor: unique}. It can be seen that $H=Z(\GL_2(\mathbb{F}_p))SL_2(\mathbb{F}_p)$. Hence it follows that $ PH \cong PSL_2(\mathbb{F}_p)$. Since $H$ is the unique index-$2$ subgroup of $\GL_2(\mathbb{F}_p)$ containing $Z(\GL_2(\mathbb{F}_p))$, we have that $PH\cong PSL_2(\mathbb{F}_p)$ is the unique index-$2$ subgroup of $\PGL_2(\mathbb{F}_p)$.\smallskip

 We have $p\geq 5$. Let $r\neq p-1$ be a non-square element in $\mathbb{F}_p $. Consider $x=\begin{pmatrix}
1 & r+1\\
-1 & -1
\end{pmatrix}\not \in H$ with $x^2=1 \ (Z(\GL_2(\mathbb{F}_p)))$.\smallskip

From \cite{arias2022locally}, we have that \cite[Theorem 1.1]{arias2022locally} is also true for $\PGL_2(\mathbb{F}_p)$. Thus we get $E_1,E_2$ satisfying conditions of Theorem \ref{thm: compositum} for $l=2$ with $G_1=G_2=\PGL_2(\mathbb{F}_p)$, and $[G_1:H_1] =2$. Therefore $H_1=H_2=PSL_2(\mathbb{F}_p)$. We observe that $\{1, xZ(\GL_2(\mathbb{F}_p))\}$ is required set of representatives of $PSL_2(\mathbb{F}_p)$-coset in $\PGL_2(\mathbb{F}_p)$. 
\end{proof}\smallskip


\begin{corollary}
    For a prime $p\geq 5$, $PSL_2(\mathbb{F}_p)$ and $PSL_2(\mathbb{F}_p)\times PSL_2(\mathbb{F}_p)$ are realizable as Galois groups over $\mathbb{Q}\sqrt{p^{\ast}}$.
\end{corollary}
\smallskip

 \section{Galois Representations, Right Splitting and Inverse Galois Problem}
\label{Gal-rep}

In this section, by using the algebraic operations induction, direct sums and tensor products on Galois representations and right splitting of some exact sequences of groups, we establish occurrence of some groups as Galois groups over $\Q$.

\smallskip

 \subsection{Induced Galois representations and Galois groups}

\begin{definition}

Let $G={\rm Gal}(\bar{\mathbb{Q}}/\mathbb{Q})$. Let $E$ be a number field and $\Lambda$ be a prime in its ring of integers $\mathcal{O}$. Let $E_{\Lambda}$ be the completion of $E$ with respect to $\Lambda$ and $\mathcal{O}_{\Lambda}$ be its ring of integers. Let $\mathbb{F}_{\Lambda}=\mathcal{O}_{\Lambda}/\Lambda \mathcal{O}_{\Lambda}$. \smallskip
We call $(\pi,V)$ a Galois representation if $V$ is a finite dimensional vector space over $E_{\Lambda}$ and $\pi:G\rightarrow \GL(V)$ is a continuous homomorphism.

\end{definition}


\begin{remark}

     Let $dim_{E_{\Lambda}}(V)=m$. Now from \cite[Proposition 9.3.5]{diamond2005first}, there is a basis of $V$ such that $\pi(G)\subset \GL_m(\mathcal{O}_{\Lambda})$ for all $g\in G$. Let us consider the maps $\pi' : G\rightarrow \GL_m(\mathbb{F}_{\Lambda})$ and $\tilde{\pi} :G\rightarrow \PGL_m(\mathbb{F}_{\Lambda})$ obtained from $\pi$ through the quotients maps $\mathcal{O}_{\Lambda}\rightarrow \mathbb{F}_{\Lambda}$ and $\GL_m(\mathbb{F}_{\Lambda})\rightarrow \PGL_m(\mathbb{F}_{\Lambda})$.

\end{remark}

\begin{definition}
     A finite group $C$ is said to be realizable as Galois group over $\mathbb{Q}$ through a Galois Representation $(\pi,V)$ if $C\cong {\rm Image}(\tilde{\pi})$.\smallskip

     We have a similar definition for $\pi'$ in place of $\tilde{\pi}$.
\end{definition}

     \medskip

  Let $K$ be a finite extension of $\mathbb{Q}$ contained in $\bar{\mathbb{Q}}$. Then $\bar{K}=\bar{\mathbb{Q}}$. Let $H={\rm Gal}(\bar{K}/K)\subset G$. Let $(\pi, W)$ be a Galois representation with $dim_{E_{\Lambda}}(W)=m$.  \smallskip

Let $\sigma=\pi|_H : H\rightarrow \GL(W)$. Then $\sigma' :H\rightarrow \GL_m ( \mathbb{F}_{\Lambda})$ and $\tilde{\sigma} :H\rightarrow \PGL_m (\mathbb{F}_{\Lambda})$ and ${\rm Image}(\tilde{\sigma})=\tilde{\pi}(H)$. Also, $[\tilde{\pi}(G):\tilde{\pi}(H)]\leq [G:H]$.\smallskip

The following discussion is valid for a general context of a finite dimensional representation of a finite index subgroup $H$ of a group $G$. In this manuscript, we only consider the case of ${\rm Gal}(\bar{\mathbb{Q}}/\mathbb{Q})$ and ${\rm Gal}(\bar{K}/K)$.\smallskip

 Let us consider the induced representation $\rho=Ind_H^G (\sigma)$ on the induced space over $E_{\Lambda}$, defined by  $$V=\{f:G\rightarrow W\ |\  f(hg)=\sigma (h)f(g)\  for\ all\ h\in H, g\in G\}$$ 
 and $\rho: G\rightarrow \GL(V)$ is given by $\rho(g)f(x)=f(xg), \text{for all}\ g,x\in G, f\in V$. Thus $(\rho, V)$ is a Galois representation.  \smallskip

Let $\{s_i\}_{1\leq i \leq n}$ be a set of representatives of right cosets in $H\backslash G$ and $\{w_j\}_{1\leq j \leq m}$ be a basis of $W$. Then we know that $\{\phi_{s_i,w_j}\}_{1\leq i \leq n,1\leq j\leq m}$ is a basis for $V$ where $$\phi_{s,w} (g)=  \begin{cases}
      \sigma(gs^{-1})w, & \text{if}\ gs^{-1}\in H \\
      0 & \text{otherwise.}
    \end{cases}$$ and $dim(V)=dim(W)[G:H]$.\smallskip

Let $H$ be a normal subgroup of $G$ (i.e., $K$ is Galois over $\Q$) such that $[G:H]=n $. Since $dim(W)=m$, we have $dim(V)=mn$. Without loss of generality, let $s_1=1$. We label the basis of $V$ as 
\[
\begin{split} 
f_1=\phi_{1,w_1},\ f_2=\phi_{1,w_2},\ \dots, \ f_m=\phi_{1,w_m}, \\
 f_{m+1}=\phi_{s_2,w_1},\ \dots,\ f_{2m}=\phi_{s_2,w_m},\dots, \\
 f_{(n-1)m+1}=\phi_{s_n,w_1},\ \dots,\ f_{nm}=\phi_{s_n,w_m}.
\end{split}
\] 
    
 We write down the matrix of $\rho(g)$ with respect to the above basis.  Observe that 
 $$\rho(g)f_i(x)= f_i(xg)= \sum \limits_{j=1}^{mn} a_{ji} f_j(x)$$
 
  By taking $x=s_k$, we get that  \smallskip
    
    $$\rho(g)=
\begin{bmatrix}
(f_1(g)) & (f_2(g)) & \dots & (f_{nm}(g)) \\
(f_1(s_2 g)) & (f_2(s_2 g)) &  \dots & (f_{nm}(s_2 g)) \\
\dots &   \dots  &   \dots  &   \dots \\
(f_1(s_n g)) & (f_2(s_n g)) & \dots & (f_{nm}(s_n g)) 
\end{bmatrix}_{nm\times nm}$$ where $(f_i(s_k g))$ are treated as $m\times 1$ 
column matrices (matrices with respect to given basis $\{w_j: 1 \leq j \leq m \}$ of $W$). \smallskip

The $nm\times nm$ matrix for $\rho (hs_i)$:  The $p,q$-th $m\times m$ block where $1\leq p,q \leq n$ and $q$ is such that $s_ps_i\in s_qH$ is $\pi (s_phs_is_q^{-1})$, since $\phi_{s_q,w}(s_phs_i)=\pi(s_ph s_is_q^{-1})w$, since $s_phs_i\in s_qH$.  \smallskip

Since $\pi(G)\subset \GL_m(\mathcal{O}_{\Lambda})$, we conclude that $\rho (G)\subset \GL_{nm} (\mathcal{O}_{\lambda})$ and $\rho':G\rightarrow \GL_{nm}(\mathbb{F}_{\Lambda})$ and $\tilde{\rho}:G\rightarrow \PGL_{nm}(\mathbb{F}_{\Lambda})$ can be defined.

    \begin{remark}
        All the following results for the pair ($\rho$, $\pi$) in this subsection also hold for pairs ($\rho'$, $\pi'$) as well as ($\tilde{\rho}$, $\tilde{\pi}$).
    \end{remark}

\begin{lemma}
   Images $\rho(H)$ and $\pi(H)$ are isomorphic.
   
\end{lemma}

\begin{proof}

 If $g=h\in H$, then $$\rho(h)=
\begin{bmatrix}
\pi(h) & 0 & \dots &   \\
0 & \pi(s_2hs_2^{-1}) & & \\
\vdots & & \ddots & \\ 
 & & & \pi(s_n hs_n^{-1})\\
\end{bmatrix}_{nm\times nm}.$$  \smallskip

    Define a map from $\rho (H)$ to $\pi (H)$ by first block projection sending $\rho(h)$ to $\pi (h)$. 
\end{proof}

\begin{lemma} \label{lemma-exact}
We get an exact sequence $$1\rightarrow \rho(H)\rightarrow \rho(G) \rightarrow G/H \rightarrow 1.$$
\end{lemma}

\begin{proof}

If $g=hs_j \in Hs_j$,  
$(f_{(j-1)m+1}(g))=\pi(h)w_1, (f_{(j-1)m+2}(g))=\pi(h)w_2,\dots ,(f_{jm}(g))=\pi(h)w_m$ and for other $k$, $(f_{k}(g))=0$. We have
    $$\rho(G)=\rho(H)\bigsqcup  \rho(Hs_2) \bigsqcup \dots \bigsqcup \rho(Hs_n)=\rho(H)\bigsqcup \rho(H)\rho(s_2) \bigsqcup \dots \bigsqcup \rho(H)\rho(s_n)$$ hence $[\rho(G): \rho(H)]=n$.  \smallskip

Now, since $H$ is normal in $G$, its left and right cosets coincide. Hence we can define a map $\gamma : \rho(G) \rightarrow G/H $ with $\gamma (\rho(hs_k))=s_kH$ for any $h,s_k$. Now $\gamma$ is well defined surjective homomorphism because $\rho(Hs_i)$ for $1\leq i\leq n$ are disjoint. The kernel of $\gamma$ is precisely $\rho (H)$.
\end{proof}

\begin{remark}
    We do not necessarily get a similar exact sequence involving $\pi(H)$, $\pi(G)$ and $G/H$ a similar well defined surjective map from $\pi(G)$ to $G/H$ may not exist since $\pi(Hs_i)$ need not be disjoint. 
    
\end{remark}

\begin{example}
Let $G/H$ be cyclic with representatives of $H$-cosets in $G$ of the form $\{1,s,s^2,\dots , s^{n-1}\}$. 
If $g=hs^i\in H$ for $0\leq i \leq n-1$, the matrix $\rho(hs^i)$ is given by  \smallskip

$$\begin{bmatrix}
0 &\dots & & 0 & \pi(h) & 0 & \dots &  \\
\vdots &  & & & 0 & \pi(shs^{-1}) & 0 & \dots\\
 & \ddots &  & & \vdots & & \ddots  & \\ 
 & & & & & & & \pi(s^{n-1-i} hs^{-(n-1-i)})\\
\pi(s^{n-i}hs^i) & & & & & & & 0 \\
0 & \pi(s^{n-(i-1)}hs^{(i-1)}) & & & & & & \vdots \\
\vdots & 0 & \ddots & & & & &  \\
 & & 0 & \pi(s^{n-1}hs) & & & &  \\
\end{bmatrix}.$$   \smallskip

In particular,  \smallskip
$$\rho(s)=
\begin{bmatrix}
0 & I_2 & 0 & \dots &  \\
0 & 0 & I_2 & 0 & \dots\\
\vdots & \vdots & & \ddots  & \\ 
 & & & & I_2 \\
\pi(s^{n}) & & & & 0 \\
\end{bmatrix}.$$
\end{example}
\smallskip

By computing and comparing the matrices $\rho(s_i),\rho(s_j)$, $\rho(s_k)$ and $\rho(s_i)\rho(s_j)$, we get the following.

\begin{lemma}
    Suppose we have a set of representatives of $H$-cosets in $G$, $\{s_1,s_2,\dots , s_n\}$ with $s_1\in H$, then $\rho(s_i)\rho(s_j)=\rho(s_k)
\iff
(\pi(s_i)\pi(s_j)=\pi(s_k)$ and $s_is_j\in s_kH$).
\end{lemma}

\begin{remark}
    $\pi(s_i)$ need not be distinct even though $\rho(s_i)$ are distinct.
    In fact even if all $\pi(s_i)$ are same, $\rho(s_i)$ will be distinct.
\end{remark}

\begin{corollary}
Let $G/H$ be cyclic with representatives of $H$-cosets in $G$ of the form $\{1,s,s^2,\dots , s^{n-1}\}$. Then $\rho(s)^n=1$ if and only if $\pi(s)^n=1$.    
\end{corollary}

\begin{theorem}
\label{thm:induced}
    Let $G$ and $H$ be as above, then the exact sequence 
    \begin{equation} \label{exact seq}
    1\rightarrow \rho(H)\rightarrow  \rho(G) \rightarrow G/H \rightarrow 1 \end{equation}
    
    is right split (that is $\rho(G)\cong \rho(H) \rtimes G/H$ for some semidirect product group law) if and only if there exists a set of representatives of $H$-cosets in $G$, $\{s_1,s_2,\dots , s_n\}$ such that $\{\rho(s_i)\}_i$ forms a multiplicatively closed subset of $\rho(G)$. (In fact it is a subgroup with same group structure as $\{s_i H\}_i= G/H$)
\end{theorem}

\begin{proof}

 Let $\{r_1,r_2,...,r_n\}$ be a set of representatives of $H$-cosets in $G$. If the exact sequence is right split, let $\iota$ be splitting with $\gamma \circ \iota = id_{G/H}$. Hence $\gamma \circ \iota (r_iH) = r_iH$ for each $i$. Thus for each $i$, $\iota (r_iH)=\rho(h_ir_i)$ for some $h_i\in H$. Let $s_i=h_ir_i$ for each $i$. Hence $\{s_1,s_2,...,s_n\}$ also forms a set of representatives for $H$-cosets in $G$ with $s_iH=r_iH$ and $\iota (s_iH)=\rho(s_i)$ for each $i$. Since $\iota$ is a homomorphism, $\{\rho(s_i)\}_i$ forms a multiplicatively closed subset of $\rho(G)$.  \smallskip
    
   Conversely, suppose there exists a set of representatives of $H$-cosets in $G$, $\{s_1,s_2,\dots , s_n\}$ with $s_1\in H$ and $\{\rho(s_i)\}_i$ forming a multiplicatively closed subset of $\rho(G)$. Hence $\{\rho(s_i)\}_i$ forms a subgroup of $\rho(G)$. $\{\rho(s_i)\}_i$ is a group with the same group structure as $\{s_i H\}_i= G/H$. Then we can define $\iota(s_i H)=\rho(s_i)$. Hence $\iota$ is a homomorphism. Also, $\gamma(\iota(s_i H))=\gamma(\rho(s_i))=s_i H$ hence $\gamma \circ \iota = id_{G/H}$. Hence the exact sequence is right split. 
\end{proof}

\begin{remark}
    We can similarly prove the statements in Remark \ref{remark proof} in previous section. 
 
\end{remark}

\begin{corollary}

   If there exists a set of coset representatives of $H$ in $G$ which forms a multiplicatively closed set, then the exact sequence \ref{exact seq} is right split.
\end{corollary}

\begin{corollary}
    Let $K$ be a cyclic extension  of $\mathbb{Q}$ and let $G$ and $H$ be as above. Then the exact sequence \ref{exact seq} is right split if and only if there exists a set of representatives of $H$-cosets in $G$ of the form $\{1,s,s^2,\dots , s^{n-1}\}$ with $\rho(s)^n =1$. Also, if these statements are true then $order(\rho(s))=n$ and $s^n\in H$.
\end{corollary}

\begin{proof}
From the above theorem, we get coset representatives of $H$, $\{s_i\}_i$ with $\{\rho(s_i)\}_i$ forming a subgroup with same group structure as $\{s_i H\}_i= G/H = <sH>$ for some $s\in G$ with $s_1\in H$. Hence without loss of generality, let $s_i H = s^{i-1} H$ for all $i$. Now let $r_i=s_2^{i-1}$. Hence $r_i H=s_i H$. Now since $s_2^n H= (s_2 H)^n=H$, we have $\tilde{\rho}(s_2)^n=1$. Hence we are done. Other assertions are clear.\end{proof}

Applying above discussion to $\tilde{\rho}$ and $\tilde{\pi}$ we have the following.

\begin{theorem}
\label{thm: M}
 Let $H$ be a finite index normal subgroup of $G=Gal(\bar{\mathbb{Q}}/\mathbb{Q})$. Suppose there exists a set of coset representatives of $H$ in $G$ 
 which form a multiplicatively closed subset of $G$. If a finite group $M$ is realizable as a Galois group over $\mathbb{Q}$ through a Galois Representation $(\pi, W)$ such that $M\cong \tilde{\pi}(G)=\tilde{\pi}(H)$, then $M \rtimes G/H$ is realisable as a Galois group over $\mathbb{Q}$ (for semidirect product group law as in Theorem \ref{thm:induced}).
\end{theorem}

Observation : $G=H\iff ker(\pi)\subset H \ and \ \pi(G)=\pi(H).$\smallskip

We can generalize Lemma ~\ref{lemma-exact} and Theorem~\ref{thm:induced}. Consider a closed subgroup $H'$ of $G$. We have $H'/(H'\cap H)\xhookrightarrow{} G/ H$. \bigskip

\bigskip

\bigskip

\begin{theorem} \label{general1} \hfill 
    
\begin{enumerate}
  \item We get an exact sequence \begin{equation}\label{general exact}   
  1\rightarrow \rho(H'\cap H)\rightarrow \rho(H') \rightarrow H'/(H'\cap H) \rightarrow 1.\end{equation}
   
\item
The exact sequence \ref{general exact} is right split (that is $\rho(H')\cong \rho(H'\cap H) \rtimes H'/(H'\cap H)$ for some semidirect product group law) if and only if there exists a set of representatives of $(H'\cap H)$-cosets in $H'$, $\{s_1,s_2,\dots , s_l\}$ such that $\{\rho(s_i)\}_i$ forms a multiplicatively closed subset of $\rho(H')$.    

    \item $H\subset H'$ if and only if $\rho(H)\subset \rho(H')$ and $ker(\pi)\cap H\subset H'$

    \item Let $H\subset H'$. Then the exact sequence \ref{general exact} is $$1\rightarrow \rho(H)\rightarrow \rho(H') \rightarrow H'/H \rightarrow 1.$$ If there exists a set of representatives of $H$-cosets in $G$, $\{s_1,s_2,\dots , s_n\}$ such that $\{\rho(s_i)\}_i$ forms a multiplicatively closed subset of $\rho(G)$, then this exact sequence is right split. That is $\rho(H')\cong \rho(H) \rtimes H'/H$ for some semidirect product group law.

    \end{enumerate}
\end{theorem}
\medskip

\subsection{Galois representations for newforms and application to the cases in work of Zywina}\label{newforms}

Let $p\geq 5$. Fix non-CM newform $f$ of weight $k > 1$ on $\Gamma_1(N)$. Let $l$ be rational prime as in \cite{zywina2023modular}. Let $\pi$ be the Deligne representation associated with $f$ with properties as in \cite[Section 1.1]{zywina2023modular}. Let $W=E_{\Lambda}^{2}$ and $\pi:G\rightarrow \GL(W)$ with $\pi(G)\subset \GL_2(\mathcal{O}_{\Lambda})$ and $\tilde{\pi} :G\rightarrow 
\PGL_2(\mathbb{F}_{\Lambda})$.  \smallskip

Let $p=l$ and assume conditions of  \cite[Section 1.2]{zywina2023modular}. Then ${\rm Image}(\tilde{\pi})=\tilde{\pi}(G)=
\PSL_2(\mathbb{F}_p)$. 
The representations $\sigma=\pi|_{H}: H\rightarrow \GL(W)$ and  $\tilde{\sigma} :H\rightarrow \PGL_2(\mathbb{F}_{\Lambda})$ are defined, and ${\rm Image}(\tilde{\sigma})=\tilde{\pi}(H)\subset \tilde{\pi}(G)$.\smallskip

\begin{lemma}\label{lemma}
    Let $p\geq 5$ and $H$ be normal in $G$ (i.e. K is Galois over Q) and $[G:H]=n < |\PSL_2(\mathbb{F}_p)|$. Then $\tilde{\pi}(H)=\PSL_2(\mathbb{F}_p)$.
\end{lemma}

\begin{proof}
Now, $[\PSL_2(\mathbb{F}_p):\tilde{\pi}(H)]= [\tilde{\pi}(G) :\tilde{\pi}(H)] \leq [G:H] = n < |\PSL_2(\mathbb{F}_p)|$. Hence $\tilde{\pi}(H)$ is not trivial. Since $H$ is normal in $G$, $\tilde{\pi}(H)$ is normal in $\tilde{\pi}(G)$. Now since $\PSL_2(\mathbb{F}_p)$ is simple we have that $\PSL_2(\mathbb{F}_p) = \tilde{\pi}(H)$ otherwise $\tilde{\pi}(H)$ will become non-trivial normal subgroup in $\PSL_2(\mathbb{F}_p)$.
\end{proof}\smallskip

\begin{remark}
    If $n \geq |\PSL_2(\mathbb{F}_p)|$ then either $\tilde{\pi}(H)=\PSL_2(\mathbb{F}_p)$ or $\tilde{\pi}(H)$ is trivial. If latter case happens, then $\tilde{\rho}(H)$ is also trivial. Hence $\tilde{\rho}(G)\cong G/H$ in that case.
\end{remark}\smallskip
 \begin{proposition}\label{PGL-PSL}
    \hfill
    \begin{enumerate}

    \item For $p\geq 5$, $\PGL_2(\mathbb{F}_p) \cong \PSL_2(\mathbb{F}_p) \rtimes \mathbb{Z}/2\mathbb{Z}$ for some semidirect product group law and the semidirect product is not direct product.
        \item Let $p\geq 5$. For any semidirect product group law such that $\PGL_2(\mathbb{F}_p) \cong \PSL_2(\mathbb{F}_p) \rtimes \mathbb{Z}/2\mathbb{Z}$, the semidirect product is not direct product.
        Furthermore, the automorphism of $\PSL_2(\mathbb{F}_p)$, given by conjugation by image of generator of $\mathbb{Z}/2\mathbb{Z}$ in $\PGL_2(\mathbb{F}_p)$, is not inner. 
    \end{enumerate}
\end{proposition}

\bigskip

\bigskip

\bigskip
 
 \begin{theorem}
     \label{thm : zywina case}
\hfill     
\begin{enumerate}
    \item  For $p\geq 5$, $\PSL_2(\mathbb{F}_p)\rtimes \mathbb{Z}/2\mathbb{Z}$ is realisable as Galois Group over $\mathbb{Q}$ (for semidirect product group law in Theorem \ref{thm:induced}).

     \item This semidirect product in part (1) is direct $\iff$ $\tilde{\pi}(s)= I$.

     \item Automorphism $\phi_{\tilde{\rho}(s)}$ of $\tilde{\rho}(H)$, by conjugation by $\tilde{\rho}(s)$, is inner.

      \item The group obtained here is not isomorphic to $\PGL_2(\mathbb{F}_p)$.

     \end{enumerate}
  \end{theorem}

 \begin{proof}
     (1) Let $[G:H]=2$. Then $H$ is normal in $G$ ($K$ is Galois). Also $2 < |\PSL_2(\mathbb{F}_p)|$.
      Hence by Lemma \ref{lemma}, $\tilde{\pi}(H)=\PSL_2(\mathbb{F}_p)$.  \smallskip

 We could have chosen $K=\mathbb{Q}(i)$ and $s\in G$ as complex conjugation so that $s^2=1$ and $\{1,s\}$ become representatives of right cosets of $H$ in $G$. Then by Theorem \ref{thm: M} we are done.  \smallskip

 (2) Now, $\tilde{\pi}(G)=\tilde{\pi}(H)=\PSL_2(\mathbb{F}_p)$.  \smallskip

Above semidirect product is direct.

     $\iff$ 
     $\tilde{\rho}(s)$=
$\begin{bmatrix}
0 & I_2 \\
I_2 & 0
\end{bmatrix}$ commutes with every $\tilde{\rho}(h)=$ $\begin{bmatrix}
\tilde{\pi}(h) & 0 \\
0 & \tilde{\pi} (shs^{-1})
\end{bmatrix}$. 

$\iff$
$\tilde{\pi}(s)$ commutes with every $\tilde{\pi}(h)$. 

$\iff$ $\tilde{\pi}(s)\in Z(\PSL_2(\mathbb{F}_p))$.

$\iff $  $\tilde{\pi}(s)= I$.

Since $Z(\PSL_2(\mathbb{F}_p))$ is trivial as $\PSL_2(\mathbb{F}_p)$ is non-abelian and simple.  \smallskip

(3) Now $\tilde{\pi}(G)=\tilde{\pi}(H)$. Hence $\tilde{\pi}(s)=\tilde{\pi}(h')$ for some $h'\in H$.

So, $\phi_{\tilde{\rho}(s)}(\tilde{\rho}(h))=\tilde{\rho}(s) \tilde{\rho}(h) \tilde{\rho}(s)^{-1} =$ $\begin{bmatrix}
\tilde{\pi}(shs^{-1}) & 0 \\
0 & \tilde{\pi} (h)
\end{bmatrix}=$ $\tilde{\rho}(h') \tilde{\rho}(h) \tilde{\rho}(h')^{-1}$.\smallskip

(4) This follows from (3) and Proposition~\ref{PGL-PSL}.\end{proof}
\smallskip

We have results similar to  (1), (2) and (3) of Theorem \ref{thm : zywina case} for certain simple groups of the form $\PSL_2(\mathbb{F}_q)$.
\begin{theorem} \label{genral q}
   The group $\PSL_2(\mathbb{F}_q)\rtimes \mathbb{Z}/2\mathbb{Z}$ is realisable as Galois Group over $\mathbb{Q}$ (for semidirect product group law in Theorem \ref{thm:induced}) for following $q$.
   \begin{enumerate}
       \item $q=p$ for $p\geq 5$ (From \cite[Section 1.2]{zywina2023modular}).

       \item $q=p^2$ for $p\equiv \pm 2 \ mod \ 5,\ p\geq 7$ (From \cite[Corollary 3.6]{dieulefait1998galois}).
       
  \item $q=p^2$ for $p\equiv \pm 3 \ mod \ 8,\ p\geq 5$ (From \cite[Corollary 3.8]{dieulefait1998galois}).
       
       \item $q=p^3$ for odd prime $p\equiv \pm 2, \pm 3, \pm 4, \pm 6 \ mod \ 13$ (From \cite[Section 1.3]{zywina2023modular}).
       
       \item $q=5^3,3^5,3^4$ (From \cite[Section 2.2, 2.5, 3]{dieulefait2022seven}).\smallskip   
       
   \end{enumerate}
\end{theorem}
\smallskip

\begin{remark}
    One has even more general conditions for $q=p^2$ (See \cite[Theorem 3.1]{dieulefait2000projective}) and $q=p^4$ (See \cite[3.3]{dieulefait2000projective}).

\end{remark}
 \smallskip

\subsection{Direct sum / tensor product of representations and Galois groups  }

\begin{lemma}

   If for each $1\leq i\leq n$, finite group $H_i$  is realizable as a Galois group over $\mathbb{Q}$
    through Galois representation $(\pi_i,V_i)$ for $m_i$-dimensional vector spaces $V_i$ over $E_{\Lambda}$, then
   $\{(\tilde{\pi}_1(g),...,\tilde{\pi}_n(g))\ |\ g\in G\}\subset H_1\times H_2 \times \dots \times 
    H_n$ is realisable as Galois group over $\mathbb{Q}$. 

\end{lemma} 

\begin{proof}
   We have $\pi_i:G\rightarrow \GL(V_i)$ with $\pi_i (G)\subset \GL_{m_i}(\mathcal{O}_{\Lambda})$
    and $\tilde{\pi_i} :G\rightarrow \PGL_{m_i}(\mathbb{F}_{\Lambda})$ such that $H_i\cong {\rm Image}(\tilde{\pi_i})$. Consider Galois representations $(\pi_i, V_i): 1 \leq i \leq n$ and their direct sum Galois representation 
 $\pi = \bigoplus \limits_{1 \leq i \leq n} \pi_i: G \rightarrow \GL (\bigoplus \limits_{1 \leq i \leq n} V_i)$. Now, $ \pi(g)={\rm diag}(\pi_1 (g),\  \dots ,\ \pi_n(g))_{m\times m}$ where $m=m_1+m_2+...+m_n$. Hence $\pi(G)\subset \GL_m (\mathcal{O}_{\Lambda})$ and $\tilde{\pi} : G\rightarrow \PGL_m(\mathbb{F}_{\Lambda})$ with 
${\rm Image}(\tilde{\pi} )\cong \{(\tilde{\pi}_1(g),...,\tilde{\pi}_n(g))\ |\ g\in G\}\subset H_1\times H_2 \times \dots \times H_n$.\end{proof}

\smallskip

We consider the tensor product Galois representation $\Pi =\bigotimes  \limits_{1 \leq i \leq n} \pi_i$ of $G$ on 
$\bigotimes  \limits_{1 \leq i \leq n} V_i$.

\begin{proposition} \label{tensor direct sum}
The groups realized as Galois groups over $\mathbb{Q}$ through Galois representations $\pi$ and $\Pi$ are isomorphic, i.e.,
\[ {\rm Image}(\tilde{\pi} )\cong {\rm Image}(\tilde{\Pi}) .\]
\end{proposition}

\begin{proof}

Let $V_1$ and $V_2$ be finite dimensional vector spaces over a general field $\F$ with dimensions $m_1$ and $m_2$ respectively.
 Given $T_i\in \GL(V_i)$ for each $i =1,2$, we consider the natural $\F$-linear map $(T_1\otimes T_2) : V_1\otimes V_2 \rightarrow V_1\otimes V_2$ that is also invertible. Fix bases of $V_1$ and $V_2$ and the corresponding basis of $V_1\otimes V_2$. Let $A_1, A_2, A$ be the matrices representing 
$T_1, T_2, T_1\otimes T_2$, respectively with respect to these bases. Then 
\[ A = A _1 \otimes A_2 =
 \begin{bmatrix}
a_{11}A_2 &\dots & a_{1m_{1}}A_2 \\
\vdots & \ddots &\vdots \\
a_{m_{1}1}A_2 & \dots & a_{m_{1}m_{1}}A_2
\end{bmatrix}_{m_1 m_2}, \]  
where $A_1= \begin{bmatrix}
a_{11} &\dots & a_{1m_1} \\
\vdots &\ddots &\vdots \\
a_{m_1 1} & \dots & a_{m_1 m_1}
\end{bmatrix}$ and $A_2= \begin{bmatrix}
a'_{11} &\dots & a'_{1m_2} \\
\vdots &\ddots &\vdots \\
a'_{m_2 1} & \dots & a'_{m_2 m_2}
\end{bmatrix}$.
\medskip
 
 If  the map $T_1 \otimes T_2$ is given by the scalar multiplication by $\lambda \in \F^\times$, then it follows that $A_i = \mu_i I$ for 
 some scalars $\mu_1, \mu_2 \in \F^\times$ such that $\lambda = \mu_1 \mu_2$. Thus the map $T_1 \otimes T_2$ descends to an injective group homomorphism
 \[  \tilde{\tau} : \PGL(V_1)\times \PGL(V_2)\rightarrow \PGL(V_1\otimes V_2).\smallskip
 \]

We take the special cases $\F = E_{\Lambda}$ and then $\F=\mathbb{F}_{\Lambda}$ to complete the proof. \smallskip
 
 Let $\F = E_{\Lambda}$. Then $\Pi(g) = \pi_1(g)\otimes  \pi_2(g)\otimes  \dots \otimes  \pi_n(g) $. Hence 
 $\Pi(G)\subset \GL_M (\mathcal{O}_{\Lambda})$ where $M=m_1 m_2\cdots m_n$ and $\tilde{\Pi} : G\rightarrow \PGL_M(\mathbb{F}_{\Lambda})$. \smallskip

Let $\F=\mathbb{F}_{\Lambda}$.
We observe that 
$\tilde{\tau}((\tilde{\pi}_1(g),\tilde{\pi}_2(g)))= \tilde{\pi_1 (g) \otimes \pi_2(g)} = (\tilde{\pi_1 \otimes \pi_2})(g)$. Hence, $\tilde{\tau}(\{(\tilde{\pi}_1(g),\tilde{\pi}_2(g))\ |\ g\in G\})={\rm Image} (\tilde{\pi_1 \otimes \pi_2})$. Since $\tilde{\tau}$ is injective, we get $\{(\tilde{\pi}_1(g),\tilde{\pi}_2(g))\ |\ g\in G\}\cong {\rm Image}(\tilde{(\pi_1 \otimes \pi_2)} )$. \smallskip

By induction on $n$, we get a well defined injective homomorphism from $ \PGL(V_1)\times \PGL(V_2)\times \dots \times \PGL(V_n)$ to $\PGL(\bigotimes \limits_{1 \leq i \leq n} V_i)$ such that ${\rm Image}(\tilde{\Pi} )\cong \{(\tilde{\pi}_1(g),...,\tilde{\pi}_n(g))\ |\ g\in G\}$.\end{proof}

\medskip

Let $L=\{(\pi(g),\rho(g))|g\in G \}$. Thus  $\tilde{L}=\{(\tilde{\pi}(g),\tilde{\rho}(g))|g\in G \}$ is realisable as a Galois group over $\mathbb{Q}$ through the Galois representation $(\pi\oplus \rho, W\oplus V)$ where $\rho$ is induced representation as before.

\begin{remark}
        All the following results for the pair ($\rho$, $\pi$) also hold for pairs ($\rho'$, $\pi'$) as well as ($\tilde{\rho}$, $\tilde{\pi}$).
    \end{remark}

\begin{lemma}\label{projection kernels}

For $H\unlhd G$ with $[G:H]=n$, we define the projection map $\Psi :L\rightarrow \rho(G)$ by $\Psi((\pi(g),\rho(g)))=\rho(g)$ and the projection map $\Phi :L\rightarrow \pi(G)$ by $\Phi((\pi(g),\rho(g)))=\pi(g)$. Then\smallskip

\begin{enumerate}
    \item The map $\Psi$ is an isomorphism.
 
\item We have $ker(\Phi)\cong ker(\pi)/(ker(\pi)\cap H)$ and it can be seen as a subgroup of $G/H$. 

\end{enumerate}
\end{lemma}

\begin{proof}

(1) 
    If $\rho(g)=I$, then $g\in H$. By matrix of $\rho(g)$, we get $\pi(g)=I$ since $g\in H$. Hence $\Psi$ is also injective.\smallskip

(2) Suppose $\pi(h_is_i)=\pi(h_js_j)=I_m$ for some $h_i,h_j\in H$. Then in the matrix of $\rho(h_is_i)$, the $p,q$-th $m\times m$ block matrix, where $1\leq p,q \leq n$ and $q$ is such that $s_p s_i\in s_q H$, is $ \pi(s_p s_q^{-1})$. Hence it is independent of $h_i\in H$. Let $\rho(h_is_i)=\theta_i$. Similarly, $\rho(h_js_j)=\theta_j$.  \smallskip

Let $s_is_jH=s_kH$ for some $k$. Hence for some $h_k\in H$, $\pi(h_is_i)\pi(h_js_j)=\pi(h_ks_k)$ and $\rho(h_is_i)\rho(h_js_j)=\rho(h_ks_k)$. Hence $\pi(h_ks_k)=I_m$ for above $h_k$. Let $\rho(h_ks_k)=\theta_k$. Then $\theta_i\theta_j=\theta_k$. Hence $\{\theta_i\}_{1\leq i\leq n}$ form a group with same group law as $\{(s_i H)\}_{1\leq i\leq n}$. Also, for any $i$, $\rho(h_is_i)=\theta_i$ if and only if $\pi(h_is_i)=I_m$. Let $J=\{i\in \{1,2,...,n\}\ |\  ker(\pi)\cap  H s_i \neq \emptyset \}$. Hence $ker(\Phi) =\{(I,\theta_i)\}_{i\in J}$. Thus we have $\Omega : ker(\Phi) \xhookrightarrow{} G/H$ sending $(I,\theta_i)$ to $s_i H$. \smallskip

 We have the usual maps $ker(\pi)\xhookrightarrow{} G\rightarrow G/H$. Hence $ker(\pi)/(ker(\pi)\cap H)\xhookrightarrow{} G/H$. Also, $ker(\pi)= \bigsqcup_{i\in J} (ker(\pi)\cap H s_i)$. Now for $i\in J$ there exists an $h_i\in H$ such that $\pi(h_is_i)=I_m$. Hence for $i\in J$, $(ker(\pi)\cap H s_i)=(ker(\pi)\cap H) (h_is_i) $. Hence $ker(\pi)= \bigsqcup_{i\in J} (ker(\pi)\cap H) h_is_i$. Thus $$ker(\pi)/(ker(\pi)\cap H) \cong \{(s_i H)\}_{i\in J} \cong ker(\Phi).$$ \vspace{-1.2 cm}
 
 \end{proof}
 
\smallskip

\begin{example}
Let $G/H$ be cyclic with representatives of $H$-cosets in $G$ of the form $\{1,s,s^2,\dots , s^{n-1}\}$. 
    Suppose $\pi(hs^i)=I_m$. Then, 
    
$$ \rho(hs^i) = \begin{bmatrix}
0 &\dots & & 0 &\pi(s^{-i}) & 0 & \dots &  \\
\vdots &  & & & 0 & \pi(s^{-i}) & 0 & \dots\\
 & \ddots &  & & \vdots & & \ddots  & \\ 
 & & & & & & & \pi(s^{-i})\\
\pi(s^{n-i}) & & & & & & & 0 \\
0 & \pi(s^{n-i}) & & & & & & \vdots \\
\vdots & 0 & \ddots & & & & &  \\
 & & 0 & \pi(s^{n-i}) & & & &  \\
\end{bmatrix}=\theta^i.$$
\end{example}
\medskip

\begin{corollary}
\label{iff}
    $\tilde L\cong \tilde\pi(G)\iff ker(\tilde\pi)\subset H$. 
\end{corollary}

\begin{proof}
    Now $\tilde L/ker(\tilde \Phi)\cong \tilde\pi(G)$ where $\tilde{\Phi} :\tilde{L}\rightarrow \tilde{\pi}(G)$ is the projection map. Since $\tilde{L}$ is finite group, we have $\tilde L\cong \tilde\pi(G)\iff ker(\tilde\Phi)$ is trivial $\iff ker(\tilde\pi)=(ker(\tilde\pi)\cap H) \iff ker(\tilde\pi)\subset H$.\end{proof}

\smallskip

By Lemma \ref{projection kernels}, $ker(\Phi)$ can be realized as
a subgroup of $G/H$ and we denote this embedding by $\Omega : ker(\Phi) \xhookrightarrow{} G/H$. Let $G'$ be the unique subgroup of $G$ such that $H \subset G' \subset G$ and 
$ker(\Phi)= G'/H $.
\smallskip

In fact, we have a more precise statement.

\begin{lemma}
\label{pi G ker phi}
Let $H\subset G' \subset G$. Then\smallskip

 $ker(\Phi)= G'/H $ under $\Omega$ $\iff$
$G'$ is the largest subgroup of $G$ such that
    $\pi(G')=\pi(H)$.\smallskip
 
     In particular, $ker(\Phi)= G/H $ under $\Omega$ $\iff$  $\pi(G)=\pi(H)$. 
\end{lemma}

\begin{proof}

Now $H\subset G' \subset G$. We have $G'= \bigsqcup_{i\in I} (G'\cap H s_i)$ where $I=\{i\in \{1,2,...,n\}\ |\  G'\cap Hs_i\neq \emptyset \} $. Also $(G'\cap Hs_i)\neq \emptyset\iff s_i\in G'$.
Hence $I=\{i\in \{1,2,...,n\}\ |\  s_i\in G'\}$. For $i\in I$, $(G'\cap H s_i)= H s_i$. Hence $G'=\bigsqcup_{i\in I} Hs_i$  and $G'/H =\{s_i H\}_{i\in I}$. Now the following  argument completes the proof. 
\[
\begin{split}
& ker(\Phi)= G'/H \ \text{under}\ \Omega \\ 
\iff & \{s_i H\}_{i\in J}=\{s_i H\}_{i\in I}, \text{that is}\  I=J \\
\iff & s_i\in G' \ \text{iff} \ ker(\pi)\cap  H s_i \neq \emptyset\\
\iff & s_i\in G'\ \text{ iff}\ \pi(s_i)\in \pi(H) \\
\iff & g\in G'\ \text{ iff }\ \pi(g)\in \pi(H)\\
\iff & G'\ \text{ is the largest subgroup of}\  G\ \text{such that}
    \  \pi(G')=\pi(H).
    \end{split}
    \]\vspace{-.9 cm}
    
\end{proof}

\smallskip

\begin{theorem}
\label{thm : L}    

   Suppose $\pi(G)=\pi(H)$. Then the exact sequence \begin{equation}\label{last exact}
       1\rightarrow ker(\Phi)  \xhookrightarrow{} L \rightarrow \pi(G)\rightarrow 1 \end{equation} is right split and $L\cong  \pi(G) \ltimes G/H $ for some semidirect product group law. 
\end{theorem}

\begin{proof}

Since $\pi(G)=\pi(H)$, $ker(\Phi)= G/H$ under $\Omega$. Consider $\iota :\pi(G) \rightarrow L$ given by $\iota(x)=(\pi(h), \rho (h))$ where $x=\pi(h)$ for some $h\in H$. If $x=\pi(h)=\pi(h')$ for some $h,h'\in H$ then $\pi(hh'^{-1})=I$. Hence by matrix calculation above, $\rho(hh'^{-1})=I$. Hence $\rho(h)=\rho(h')$. Thus $\iota$ is well defined. Clearly, $\iota$ is a homomorphism and $\Phi \circ \iota =id_{\tilde{\pi}(G)}$, therefore the exact sequence splits.\end{proof}

\begin{corollary}
Let $G/H$ be cyclic $<s H>$. If $ker(\pi)\cap Hs \neq \emptyset$ then the exact sequence \ref{last exact} is right split.    
\end{corollary}

\begin{proof}
    Since $ker(\pi)\cap Hs \neq \emptyset$, there is an $h_1\in H$ such that $\pi(h_1 s)=1$. Since $H$ is normal in $G$, $(h_1s)^i=h_is^i$ for some $h_i\in H$. Hence $\pi(h_i s^i)=\pi(h_1 s)^i=1$. Hence for all $i$, $ker(\pi)\cap Hs^i \neq \emptyset$. Thus $ker(\Phi)\cong G/H$.\end{proof}

\begin{corollary}
    Let $\{\rho(s_i)\}_i$ form a multiplicatively closed subset of $\rho(G)$ and let $(n,|\pi(G)|)=1$. Then the exact sequence \ref{last exact} is right split.    
\end{corollary}

\begin{proof}
   Since $|G/H|=n$, for any $i$, $(s_iH)^n =H$. Thus $\rho(s_i)^n = I$. Hence $\pi(s_i)^n =I$. Since $\pi(s_i)\in \pi(G)$, we have $\pi(s_i)^{|\pi(G)|} =I$. Now $(n,|\pi(G)|)=1$. Hence, $\pi(s_i)=1$ for all $i$. Thus $\pi(H)=\pi(G)$.
\end{proof}
\smallskip

Recall the results from \cite{zywina2023modular} that we discussed earlier in Sec.~\ref{newforms}. In this case, we have \smallskip

 \begin{corollary}
     For $p\geq 5$, $\PSL_2(\mathbb{F}_p)\times \mathbb{Z}/2\mathbb{Z}$ is realisable as Galois Group over $\mathbb{Q}$. 
 \end{corollary}

 \begin{proof}
Since $\tilde{\pi}(H)=\tilde{\pi}(G)$ as in Theorem \ref{thm : zywina case},
we have from above that $\PSL_2(\mathbb{F}_p)\ltimes \mathbb{Z}/2\mathbb{Z}$ is realisable as Galois Group over $\mathbb{Q}$ (for semidirect product group law as in \ref{thm : L}). Since $Aut(\mathbb{Z}/2\mathbb{Z})$ is trivial, we have that semidirect product is indeed direct.
 \end{proof}
 
 We can generalize Lemma \ref{projection kernels}, Corollary \ref{iff}, Lemma \ref{pi G ker phi} and Theorem \ref{thm : L}. Consider a closed subgroup $H'$ of $G$. We have $H'/(H'\cap H)\xhookrightarrow{} G/ H$.

\begin{theorem}
  Let $N=\{(\pi(h'),\rho(h'))|h'\in H' \}$ and $\tilde{N}=\{(\tilde{\pi}(h'),\tilde{\rho}(h'))|h'\in H' \}$.
    \begin{enumerate}
         \item The surjective projection $N\rightarrow \rho(H')$ is an isomorphism.
 
\item Consider the surjective projection $\xi : N\rightarrow \pi(H')$. We have $\omega : ker(\xi) \xhookrightarrow{} H'/(H'\cap H)$ and  $ker(\xi)\cong (ker(\pi)\cap H')/(ker(\pi)\cap H' \cap H)$.

\item  $\tilde{N}\cong \tilde{\pi}(H')\iff ker(\tilde{\pi})\cap H' \subset H$.

\item Let $(H'\cap H) \subset G' \subset H'$ then\\
 $ker(\xi)= G'/(H'\cap H) $ under $\omega$ $\iff$
$G'$ is largest subgroup of $H'$ such that
    $\pi(G')=\pi(H'\cap H)$.
 
     In particular, $ker(\xi)= H'/(H'\cap H) $ under $\omega$ $\iff$  $\pi(H')=\pi(H'\cap H)$. 

     \item Suppose $\pi(H')=\pi(H'\cap H)$. Then the exact sequence \begin{equation}\label{really the last exact}
     1\rightarrow ker(\xi)  \rightarrow N \rightarrow \pi(H')\rightarrow 1\end{equation} is right split and $N\cong  \pi(H') \ltimes H'/(H'\cap H) $ for some semidirect product group law. 

\item  Let $H\subset H'$. Suppose $\pi(G)=\pi(H)$. Then the exact sequence \ref{really the last exact} is right split and $N\cong  \pi(H') \ltimes H'/H $ for some semidirect product group law. 
     
    \end{enumerate}
\end{theorem}
\medskip

\noindent {\it Acknowledgements:} Both the authors would like to thank Prof. Joachim K{\"o}nig for a prompt reply and useful discussion over email about the results by 
Arias-de-Reyna \& K{\"o}nig in \cite{arias2022locally}. The second author Shubham Jaiswal would like to acknowledge support of IISER Pune Institute Scholarship during this work. 	


\medskip


 \begin{thebibliography}{}

\bibitem{arias2022locally}
Arias-de-Reyna, Sara and K{\"o}nig, Joachim.~
\textit{Locally cyclic extensions with Galois group $ \GL_2 (p) $}, International Journal of
Number Theory, 20(03):781–796, 2024.

\bibitem{conrad2023galois}
Conrad, Keith.~
\textit{The Galois correspondence at work},
URL: https://kconrad.math.uconn.edu/blurbs/galoistheory/galoiscorrthms.pdf,
2023.

\bibitem{diamond2005first}
Diamond, Fred and Shurman, Jerry Michael.~
\textit{A first course in modular forms},
Springer publications, 228, 2005.


\bibitem{dieulefait2022seven}
Dieulefait, Luis, Florit, Enric and Vila, N{\'u}ria,
\textit{Seven small simple groups not previously known to be Galois over $\mathbb Q$},
Mathematics, 10(12), 2022.

\bibitem{dieulefait1998galois}
Dieulefait, Luis Victor,
\textit{Galois realizations of families of Projective Linear Groups via cusp forms},
  1998. hal-00147919f.

\bibitem{dieulefait2000projective}
Dieulefait, Luis and Vila, N{\'u}ria,
\textit{Projective Linear Groups as Galois Groups over $\Q$ via Modular Representations},
Journal of Symbolic Computation, 30(6), 2000.
  
  \bibitem{ribet}
  Ribet, Kenneth A.,
  \textit{ On  $l$-adic representations attached to modular forms. II},
Glasgow Math. J. 27 (1985), 185-194.

\bibitem{zywina2023modular}
Zywina, David.~
\textit{Modular forms and some cases of the Inverse Galois Problem}, 
Canadian Mathematical Bulletin, 66(2), 2023, 568--586.

\end{thebibliography}
\end{document}